\documentclass{article}
\usepackage{latexsym}
\usepackage{amssymb,amsthm,amsmath}
\usepackage{verbatim}

\usepackage{latexsym}
\usepackage{amssymb,amsmath}
\usepackage{amscd}

\newtheorem{theorem}{Theorem}[section]
\newtheorem{proposition}[theorem]{Proposition}
\newtheorem{lemma}[theorem]{Lemma}

\newtheorem{definition}[theorem]{Definition}

\newcommand{\be}{\begin{eqnarray}}
\newcommand{\ee}{\end{eqnarray}}
\newcommand{\ben}{\begin{eqnarray*}}
\newcommand{\een}{\end{eqnarray*}}
\newcommand{\NN}{\mathbb N}
\newcommand{\ZZ}{\mathbb Z}

\newcommand{\FF}{\mathbb F}

\def\N{\mathbf N}

\def\cD{\mathcal D}

\def\cH{\mathcal H}

\def\cX{\mathcal X}

\def\l{\ell}

\def\fq{{\mathbb F}_q}

\def\e{{\epsilon}}
\begin{document}
\title{On linear codes from maximal curves}

\author{Stefania~Fanali
\thanks{S. Fanali is with the Dipartimento di Matematica e Informatica, Universit\`a di  Perugia,
Via Vanvitelli 1, 06123, Perugia,
Italy (e-mail: stefania.fanali@dipmat.unipg.it)}
\thanks{This research was performed within the activity of GNSAGA of the
Italian INDAM.}}
\maketitle

\begin{abstract} Some linear codes associated to maximal algebraic curves via Feng-Rao construction  \cite{FR2} are investigated. In several case, these codes have better minimum distance with respect to the previously known linear codes with same length and dimension.
\end{abstract}

\section{Introduction}
The idea of constructing linear codes from algebraic curves defined over a finite field $\fq$ goes back to Goppa \cite{GOP}. These codes are usually called Algebraic Geometric Codes, AG codes for short. Typically, AG codes with good parameters arise from curves with a large number $N$ of $\fq$-rational points with respect to their genus $g$.  In fact, for an $[N,k]_q$ AG code  code associated to a curve of genus $g$,  the sum of its transmission rate plus its relative minimum distance is at least
$
1-\frac{g-1}{N}.
$
An upper bound on $N$ is given by the Hasse-Weil estimate $N\le q+1+2g\sqrt q$.

In 1995 Feng and Rao \cite{FR2} introduced the so-called Improved AG Codes, see Section \ref{improved}. The parameters of these codes depend on the pattern of the Weierstrass semigroup at the points of the underlying curve, and it has emerged that they can be significantly better than those of the ordinary AG codes.  

The aim of this paper is to  investigate the parameters of the Improved AG Codes associated to some classes of maximal curves, that is, curves for which the Hasse-Weil upper bound is attained. Since maximal curves with positive genus exist only for square $q$, henceforth  we assume that $q=q_02$. The main achievement of the paper is the discovery of several  linear codes that apparently have better parameters with respect to the previously known ones, see Appendix. Our method is based on the explicit description of the Weierstrass semigroup at some $\fq$-rational points of the curves under investigation.

The maximal curves that will be considered  are the following.
\begin{itemize}
\item[(A)] Curves with equation $X^{2m}+X^{m}+Y^{q_{0}+1}=0$, where $m>2$ is a divisor of $q_{0}+1$, and $\Delta=\frac{q_{0}+1}{m}>3$ is a prime \cite{GHKT}.
\item[(B)] Curves with equation $X^{2i+2}+X^{2i}+Y^{q_{0}+1}=0$, where $\Delta=\frac{q_{0}+1}{2}>3$ is a prime, and $1\leq i\leq \Delta-2$ \cite{GHKT2}.
\item[(C)] Quotient curves of the Hermitian curve $Y^{q_0+1}=X^{q_{0}}+X$ with respect to the additive subgroups of $H=\{c\in \fq\mid c^{q_0}+c=0\}$ \cite{GSX}.
\item[(D)] Curves with equation $Y^{m}=X^{q_{0}}+X$, where $m$ is a proper divisor of  $q_{0}+1$ \cite{CKT1}.
\item[(E)] Curves with equation $Y^{\frac{q-1}{m}}=X(X+1)^{q_{0}-1}$, where $m$ is a divisor of $q-1$ \cite{GSX}.

\end{itemize}

\section{Notation and Preliminaries}\label{sec2}

\subsection{Curves} Throughout the paper, by a curve we mean a projective, geometrically irreducible, non-singular algebraic curve defined over a finite field.  Let $q_{0}$ be a prime power, $q=q_{0}^{2}$, and let $\cX$ be a curve defined over the finite field $\FF_{q}$ of order $q$. Let $g$ be the genus of $\cX$.
Henceforth, the following notation is used:
\begin{enumerate}
\item[$\bullet$]  $\cX(\FF_{q})$ (resp. $\FF_{q}(\cX)$) denotes the set of $\FF_{q}$-rational points (resp. the field of $\FF_{q}$-rational functions) of $\cX$.
\item[$\bullet$] $\cH$ is the Hermitian curve over $\FF_{q}$ with affine equation
\begin{equation}\label{Hcurve}
Y^{q_{0}+1}=X^{q_0}+X.
\end{equation}
\item[$\bullet$] For $f\in\FF_{q}(\cX)$,  $(f)$ (resp. $(f)_{\infty}$) denotes the divisor (resp. the pole divisor) of  $f$.
\item[$\bullet$] Let $P$ be a point of $\cX$. Then $ord_{P}$ (resp. $H(P)$) stands for the valuation (resp. for the Weierstrass non-gap semigroup) associated to $P$. The $i$th non-gap at $P$ is denoted as $m_i(P)$.

\end{enumerate}

\subsection{One-point AG Codes and Improved AG Codes}\label{improved}
Let $\cX$ be a curve, let $P_{1},P_{2},\ldots,P_{n}$ be $\FF_{q}$-rational points of $\cX$, and let $D$ be the divisor $P_{1}+P_{2}+\ldots+P_{n}$. Furthermore, let  $G$ be some other divisor that has support disjoint from $D$.
The AG code $C(D,G)$ of length $n$ over $\FF_{q}$ is the image of the linear map $\alpha : L(G)\rightarrow\FF_{q}^{n}$ defined by $\alpha(f)=(f(P_{1}),f(P_{2}),\ldots,f(P_{n}))$. If $n$ is bigger than $deg(G)$, then $\alpha$ is an embedding, and the dimension $k$ of  $C(D,G)$ is equal to $\ell(G)$. The Riemann-Roch theorem makes it possible to estimate the parameters of $C(D,G)$. In particular, if $2g-2<deg(G)<n$, then $C(D,G)$ has dimension $k=deg(G)-g+1$ and minimum distance $d\geq n-deg(G)$, see e.g. \cite[Theorem~2.65]{HLP}. A generator matrix $M$ of $C(D,G)$ is
$$M=\left(\begin{array}[pos]{ccc}
f_{1}(P_{1}) & \ldots & f_{1}(P_{n}) \\
\vdots & \ldots & \vdots \\
f_{k}(P_{1}) & \ldots & f_{k}(P_{n})\\
\end{array}\right),$$
where $f_{1},f_{2},\ldots,f_{k}$ is an $\FF_{q}$-basis of $L(G)$.
The dual code $C^{\perp}(D,G)$ of $C(D,G)$ is an AG code with dimension $n-k$ and minimum distance greater than or equal to $deg(G)-2g+2$.
When $G=\gamma P$  for an $\FF_{q}$-rational $P$ point of $\cX$, and a positive integer $\gamma$, AG codes $C(D,G)$ and $C^{\perp}(D,G)$ are referred to as one-point AG codes. We recall some results on the minimum distance of one-point AG codes. By \cite[Theorem~3]{GKL}, we can assume that $\gamma$ is a non-gap at $P$. Let
$$H(P)=\left\{\rho_{1}=0<\rho_{2}<\ldots\right\},$$
and set $\rho_{0}=0$.
Let $f_{\l}$ be a rational function such that $div_{\infty}(f_{\l})=\rho_{\l}P$, for any $\l\geq1$. Let $D=P_{1}+P_{2}+\ldots+P_{n}$. Let also
\begin{equation}\label{hl}
h_{\l}=(f_{\l}(P_{1}),f_{\l}(P_{2}),\ldots,f_{\l}(P_{n}))\in\FF_{q}^{n}.
\end{equation}
%

Set
$$\nu_{\l}:=\#\left\{(i,j)\in\NN^{2} : \rho_{i}+\rho_{j}=\rho_{\l+1}\right\}$$
for any $\l\geq0$. Denote with $C_{\l}(P)$ the dual of the AG code $C(D,G)$, where $D=P_{1}+P_{2}+\ldots+P_{n}$, and $G=\rho_{\l}P$.

\begin{definition}\label{n5}
Let $d$ be an integer greater than $1$. The Improved AG code $\tilde{C}_d(P)$ is the code
$$\tilde{C}_d(P):=\left\{x\in\FF_{q}^{n} : \left\langle x,h_{i+1}\right\rangle=0 \textrm{ for all } i \textrm{ such that } \nu_{i}<d\right\},$$
see {\rm \cite[Def. 4.22]{HLP}}.
\end{definition}
\begin{theorem}[Proposition 4.23 in \cite{HLP}]\label{dist2}
Let
$$r_{d}:=\#\left\{i\geq0 : \nu_{i}<d\right\}.$$
Then $\tilde{C}_d(P)$ is an $\left[n,k,d'\right]$-code, where
$k\geq n-r_{d}$, and $d'\geq d$.
\end{theorem}

\subsection{Maximal Curves} A curve $\cX$ is called $\FF_{q}$-maximal if the number of its $\FF_{q}$-rational points attains
the Hasse-Weil upper bound, that is,
$$\#\cX(\FF_{q})=q+1+2gq_{0},$$
where $g$ is the genus of $\cX$.

A key tool for the investigation of maximal curves is Weiestrass Points theory.
The Frobenius linear series of a maximal curve $\cX$ is the complete linear series $\cD=|(q_{0}+1)P_0|$, where $P_0$ is any $\FF_{q}$-rational point of $\cX$.
The next result provides a relationship between $\cD$-orders and non-gaps at points of $\cX$.

\begin{proposition}[\cite{FGT}]\label{FGT4}
Let $\cX$ be a maximal curve over $\FF_{q}$, and let $\cD$ be the Frobenius linear series of $\cX$. Then
\begin{enumerate}
\item[(i)] For each point $P$ on $\cX$, we have $\l(q_{0}P)=r$, i.e.,
$$0<m_{1}(P)<\ldots<m_{r-1}(P)\leq q_{0}<m_{r}(P).$$
\item[(ii)] If $P$ is not rational over $\FF_{q}$, the $\cD$-orders at the point $P$ are
$$0\leq q_0-m_{r-1}(P)<\ldots<q_0-m_{1}(P)<q_0.$$
\item[(iii)] If $P$ is rational over $\FF_{q}$, the $(\cD, P)$-orders are
$$0<q_0+1-m_{r-1}(P)<\ldots<q_0+1-m_{1}(P)<q_0+1.$$
In particular, if $j$ is a $\cD$-order at a rational point $P$, then $q_0+1-j$ is a non-gap at $P$.
\item[(iv)] If $P\in\cX(\FF_{q})$, then $q_0$ and $q_0+1$ are non-gaps at $P$.
\end{enumerate}
\end{proposition}

\section{Weierstrass semigroups for  curves (A)}

Let $m>2$ be a divisor of $q_{0}+1$, and suppose that $\Delta=\frac{q_{0}+1}{m}>3$ is a prime.
Let $\cX_{m}$ be the non-singular model of the plane curve over $\FF_{q}$ with affine equation
$$X^{2m}+X^{m}+Y^{q_{0}+1}=0.$$
\begin{proposition}[Section 3 in \cite{GHKT}]\label{n1} The curve $\cX_{m}$ has the following properties.
\begin{enumerate}
\item[\rm{(i)}] The genus of $\cX_{m}$ is $g=\frac{1}{2}m(q_{0}-2)+1$.
\item[\rm{(ii)}] $\cX_{m}$  is a maximal curve with
$$q+1+m(q_{0}-2)q_{0}+2q_{0}\geq 1+qm$$
$\FF_{q}$-rational points.
\item[\rm{(iii)}] If $\omega$ is a primitive m-th root of $-1$, then there exists an $\FF_{q}$-rational point $P$ such that
\begin{itemize}
\item $(\frac{1}{{\bar x}-\omega})_\infty=(q_0+1)P$;
\item $(\frac{{\bar x}^{-1}{\bar y}^\Delta}{{\bar x}-\omega})_\infty=(q_0+1-\Delta)P$;
\item for all $n=1,\ldots,\frac{\Delta-1}{2}$, $(\frac{{\bar y}^n}{{\bar x}-\omega})_\infty=(q_0+1-n)P$.
\end{itemize}
\end{enumerate}
\end{proposition}
Let $P$ be as (iii) of Proposition \ref{n1}. Then the Weierstrass semigroup $H(P)$ contains the following numerical semigroup
$$\Theta=\left\langle q_{0}+1-\Delta, q_{0}+1-\frac{\Delta-1}{2}, q_{0}+1-\frac{\Delta-1}{2}+1, \ldots, q_{0}+1\right\rangle.$$
We show that actually $H(P)$ coincides with $\Theta$.
\begin{theorem}
$H(P)=\Theta$.
\end{theorem}
\begin{proof}
To prove the assertion we show that the number of gaps in $\Theta$, that is, the number of integers in $\N\setminus\Theta$, is equal to the genus $g$ of $\cX_{m}$.
Let $G=\left\{q_{0}+1-\Delta, q_{0}+1-\frac{\Delta-1}{2}, q_{0}+1-\frac{\Delta-1}{2}+1, q_{0}+1-\frac{\Delta-1}{2}+2,\right.$
$\left.\ldots, q_{0}+1\right\}$, and for $s\in\N_{0}$ let $G(s)=\left\{ig_{1}+jg_{2} | g_{k}\in G, i+j=s\right\}$.
Note that, if $s<m$, then the intersection of two of these sets is always empty. In fact, the largest integer of $G(s-1)$ is $(s-1)(q_{0}+1)$, and it is smaller than the smallest integer of $G(s)$, that is $s(q_{0}+1-\Delta)$.
So, the number of gaps between $G(s-1)$ and $G(s)$ is exactly $s(q_{0}+1-\Delta)-(s-1)(q_{0}+1)-1=q_{0}-s\Delta$. Now we show that the number of gaps in $G(s)$ is at most $\frac{\Delta-1}{2}$. Let $v\in G(s)$. It is easily seen that $G(s)$ contains each $v$ such that $s(q_{0}+1-\frac{\Delta-1}{2})\leq v\leq s(q_{0}+1)$.
Moreover, if $s(q_{0}+1)-s\Delta+r\Delta-r\frac{\Delta-1}{2}\leq v\leq s(q_{0}+1)-s\Delta+r\Delta$, then
$$v=(s-r)(q_{0}+1-\Delta)+(q_{0}+1-\frac{\Delta-1}{2}+h)+(q_{0}+1-\frac{\Delta-1}{2}+k),$$
where $h,k\in\left\{0,\ldots,m-1\right\}$, and $r\in\left\{0,\ldots,s-1\right\}$.
Hence, $G(s)$ contains every integer greater than $(s-1)(q_{0}+1-\Delta)+(q_{0}+1-\frac{\Delta-1}{2})$, and less than $s(q_{0}+1-\frac{\Delta-1}{2})$.
For the same reason, the number of gaps in $G(m)$ is at most $\frac{\Delta-1}{2}$.
In particular, the greatest gap in $G(m)$ is less than $(m-1)(q_{0}+1-\Delta)+(q_{0}+1-\frac{\Delta-1}{2})<2g$, and the greatest $v\in G(m)$ is such that $v>2g$.
Therefore, we have at most
$$m\frac{\Delta-1}{2}+\sum_{i=1}^{m-1}(q_{0}-i\Delta)=m\frac{\Delta-1}{2}+q_{0}(m-1)-\Delta \frac{m(m-1)}{2}=g$$
gaps less than $2g$.
This shows that $\Theta\cap\left[0, 2g\right]=H(\gamma)\cap\left[0, 2g\right]$.
To complete the proof, we need to show that $\Theta$ contains every integer greater than $2g$. This follows from the fact that if $s\geq m$, then $G(s)\cap G(s+1)\neq\emptyset$; moreover, $G(s)$ contains the gaps of $G(s+1)$, being $s(q_{0}+1)>(s+1)(q_{0}+1)-\Delta-\frac{\Delta-1}{2}$. This completes the proof.
\end{proof}

\section{Weierstrass semigroups for curves (B)}\label{curva2}

Let $q_{0}$ be a prime power such that $\Delta=\frac{q_{0}+1}{2}$ is a prime greater than $3$.
Let $\cX_{i}$ be the non-singular model of the curve over $\FF_{q}$ with affine equation
$$X^{2i+2}+X^{2i}+Y^{q_{0}+1}=0,$$
where $1\leq i\leq \Delta-2$.

\begin{proposition}[\cite{GHKT2}] The curve
$\cX_{i}$ has the following properties.
\begin{enumerate}
\item[$(i)$] $\cX_i$ is maximal.
\item[$(ii)$] The genus of $\cX_{i}$ is $g=q_{0}-1$.
\item[$(iii)$] Let $n_{1},n_{2},\ldots, n_{\frac{\Delta-1}{2}}$ be the integers such that $0<n_{j}<\Delta$, and $$n_{j}(i+1)\leq \left(\left\lfloor \frac{n_{j}i}{\Delta}\right\rfloor +1 \right)\Delta,$$ for $j\in\left\{1,2,\ldots,\frac{\Delta-1}{2}\right\}$.
\begin{enumerate}
\item[$\bullet$] There exists an $\FF_{q}$-rational point $P_{1}$ of $\cX_i$ such that the Weierstrass semigroup $H(P_{1})$ contains the following integers
$$q_{0}+1, q_{0}+1-n_{1}, q_{0}+1-n_{2}, \ldots, q_{0}+1-n_{\frac{\Delta-1}{2}}, q_{0}+1-\Delta;$$
\item[$\bullet$] there exists an $\FF_{q}$-rational point $P_{2}$ of $\cX_i$ such that the Weierstrass semigroup $H(P_{2})$ contains the following integers
$$q_{0}+1, q_{0}+1-m_{1}, q_{0}+1-m_{2}, \ldots, q_{0}+1-m_{\frac{\Delta-1}{2}}, q_{0}+1-\Delta,$$
where $m_{j}=n_{j}i-\Delta\left\lfloor \frac{n_{j}i}{\Delta}\right\rfloor$;
\item[$\bullet$] there exists an $\FF_{q}$-rational point $P_{3}$ of $\cX_i$ such that the Weierstrass semigroup $H(P_{3})$ contains the following integers
$$q_{0}+1, q_{0}+1-k_{1}, q_{0}+1-k_{2}, \ldots, q_{0}+1-k_{\frac{\Delta-1}{2}}, q_{0}+1-\Delta,$$
where $k_{j}=\Delta\left(\left\lfloor \frac{n_{j}i}{\Delta}\right\rfloor+1\right)-n_{j}(i+1)$.
\end{enumerate}
\end{enumerate}
\end{proposition}

It is easily seen that, if $i=1$, then $n_{j}=m_{j}=j$, and $k_{j}=\Delta-2j$, for $j\in\left\{1,2,\ldots,\frac{\Delta-1}{2}\right\}$. So, from the above Proposition we have just two different Weierstrass semigroup of $\cX_{1}$, namely $H(P_{1})$ and $H(P_{3})$.

\begin{theorem}\label{sgr1}
Assume that $i=1$, and let
$$\Theta=\left\langle q_{0}+1, (q_{0}+1)-1, (q_{0}+1)-2, \ldots, (q_{0}+1)-\frac{\Delta-1}{2}, q_{0}+1-\Delta\right\rangle.$$
Then $H(P_{1})=\Theta.$
\end{theorem}
\begin{proof}
To prove the assertion we show that the number of gaps in $\Theta$ is equal to the genus $g$ of $\cX_{1}$. Let $G=\left\{q_{0}+1, (q_{0}+1)-1, (q_{0}+1)-2, \ldots, (q_{0}+1)-\frac{\Delta-1}{2}, q_{0}+1-\Delta\right\}$, and for $s\in\left\{0,1,2\right\}$ let $G(s)=\left\{ig_{1}+jg_{2} | g_{k}\in G, i+j=s\right\}$.
Note that $G(0)=\left\{0\right\}$, $G(1)=G$, and that the number of gaps in $G(1)$ is at most $\frac{\Delta-1}{2}$. Moreover,
$G(1)\cap G(2)=\emptyset$. In fact, the largest integer of $G(1)$ is $q_{0}+1$, and it is smaller than the smallest integer of $G(2)$, that is $(q_{0}+1-\Delta)+(q_{0}+1-\frac{\Delta-1}{2})=\frac{5(q_{0}+1)}{4}+\frac{1}{2}$. So, the number of gaps between $G(1)$ and $G(2)$ is exactly $(q_{0}+1-\Delta)+(q_{0}+1-\frac{\Delta-1}{2})-(q_{0}+1)-1=\frac{q_{0}+1}{4}-\frac{1}{2}$.
Now we show that $\left(\N\cap\left[\frac{5(q_{0}+1)}{4}+\frac{1}{2}, 2(q_{0}+1)\right]\right)\setminus G(2)=\emptyset$. Let $v\in G(2)$.
\begin{enumerate}
\item[$\bullet$] If $\frac{5(q_{0}+1)}{4}+\frac{1}{2}\leq v\leq \frac{3(q_{0}+1)}{2}$, then $v=q_{0}+1-\Delta+(q_{0}+1-\frac{\Delta-1}{2}+h)$, where $h\in\left\{0,\ldots,\frac{\Delta-1}{2}\right\}$;
\item[$\bullet$] if $\frac{3(q_{0}+1)}{2}+1\leq v\leq 2(q_{0}+1)$, then
$v=(q_{0}+1-\frac{\Delta-1}{2}+h)+(q_{0}+1-\frac{\Delta-1}{2}+k)$, where $h, k\in\left\{0,\ldots,\frac{\Delta-1}{2}\right\}$.
\end{enumerate}
Moreover, since that the largest integer in $G(2)$, that is $2(q_{0}+1)$, is smaller than $2g=2(q_{0}-1)$, we have at most
$$(q_{0}+1-\Delta-1)+\frac{\Delta-1}{2}+(\frac{q_{0}+1}{4}-\frac{1}{2})=q_{0}-1=g$$
gaps less than $2g$.
This shows that $\Theta\cap\left[0, 2g\right]=H(\gamma_{1})\cap\left[0, 2g\right]$.
To complete the proof, we need to show that $\Theta$ contains every integer greater than $2g$. Consider now $G(s)$, for $s\geq 3$.
We can observe that in $G(s)$ there is no gap, and that between $G(2)$ and $G(3)$ there is no integer. Also, it is easily seen that for $s>2$, $G(s)\cap G(s+1)\neq\emptyset$ holds. This completes the proof.
\end{proof}

\begin{theorem}\label{sgr2}
Assume that $i=1$, and let
 $$\Gamma=\left\langle q_{0}+1-\Delta, q_{0}+1-(\Delta-2), q_{0}+1-(\Delta-4),\ldots, q_{0}, q_{0}+1 \right\rangle.$$
Then $H(P_{3})=\Gamma$.
\end{theorem}
\begin{proof}
We show that the number of gaps in $\Gamma$ is equal to the genus $g$ of $\cX_{1}$. Let $h=q_{0}+1-\Delta$, and let
$G=\left\{h, h+2, h+4, \ldots, 2h-1, 2h\right\}$. Clearly, $\Gamma$ is generated by $G$.
The number of gaps less than the first non-zero nongap is $h-1$, and the number of gaps in $G$ is at most $\frac{h-1}{2}$.
For $s\in\left\{0,1,2\right\}$ let $G(s)=\left\{ig_{1}+jg_{2} | g_{k}\in G, i+j=s\right\}$. Note that $G(0)=\left\{0\right\}$, $G(1)=G$, and that the number of gaps in $G(2)$ is at most $\frac{h-1}{2}$. In fact, for $v\in G(2)$, we have that if $v\leq 3h$, then $v=h+(h+2i)$ for some $i\in\left\{1,\ldots, \frac{h-1}{2}\right\}$, and if $3h-1<v\leq 4h$, then $v=(h+2i)+(h+2j)$ for some $i,j\in\left\{0,1,\ldots\frac{h-1}{2}\right\}$.
Since that $4h=2(q_{0}+1)>2g$, we have at most
$$2(h-1)=g$$
gaps less than $2g$. Moreover, it is easily seen that $\Gamma$ contains every integer greater than $2g$. 
\end{proof}

Consider the curve $\cX_{i}$, when $q_{0}=9$, and $\Delta=5$.
We limit ourselves to the  the case $i=1$, since for $i=2$ and $i=3$ the same Weierstrass semigroups are obtained.
The curve $\cX_{1}$ has equation
\begin{equation}\label{curva2p}
X^{4}+X^{2}+Y^{10}=0,
\end{equation}
its genus is $g=8$, and the number of its $\FF_{81}$-rational points is $226$.
By Theorems \ref{sgr1} and \ref{sgr2} there exist two $\FF_{q}$-rational points $P_1$ and $P_{3}$ such that $H(P_{1})=\left\langle 5,8,9\right\rangle$ and $H(P_{3})=\left\langle 5,7,9\right\rangle$.

In Appendix, the  Improved AG codes associated to the curve (\ref{curva2p}) with respect to $P_3$ will be referred to as codes $(B1)$.

\section{Weierstrass semigroups for  curves (C)}

Let $s$ be a divisor of $q_{0}$.
Consider $H=\left\{c\in\FF_{q} |\, c^{q_{0}}+c=0\right\}<(\FF_{q}, +)$, and let $H_{s}$ be any additive subgroup of $H$ with $s$ elements.
Let $\cX$ be the curve obtained as the image of the Hermitian curve by the following rational map
\begin{equation}\label{quo}
\varphi : \cH\rightarrow\cX,\,\,\, (x,y)\mapsto (t,z)=(\prod_{a\in H_{s}}(x+a), y).
\end{equation}

\begin{proposition}[\cite{GSX}]\label{genus} The genus
$g$ of $\cX$ is equal to $\frac{1}{2}q_{0}(\frac{q_{0}}{s}-1).$
\end{proposition}

Let $\bar{P}_{\infty}$ be the only point at infinity of $\cX$. This point is the image of the only infinite point $P_{\infty}$ in $\cH$ by $\varphi$. Hence, the ramification index $\e_{P_{\infty}}$ of $P_{\infty}$ is equal to $deg(\varphi)=s$.
Moreover, it is easily seen that
$$ord_{\bar{P}_{\infty}}(t)=\frac{1}{s}\sum_{a\in H_{s}}ord_{P_{\infty}}(x+a)=-(q_{0}+1),$$
and
$$ord_{\bar{P}_{\infty}}(z)=\frac{1}{s}ord_{P_{\infty}}(y)=-\frac{q_{0}}{s}.$$
Hence,
\begin{equation}\label{n2}
\left\langle \frac{q_{0}}{s}, q_{0}+1\right\rangle\subseteq H(\bar{P}_{\infty}).
\end{equation}

\begin{proposition}\label{sW}
$$H(\bar{P}_{\infty})=\left\langle \frac{q_{0}}{s}, q_{0}+1\right\rangle.$$
\end{proposition}
\begin{proof} Note that $\frac{q_0}{s}$ and $q_0+1$ are coprime. Then by \cite[Proposition 5.33]{HLP} the genus of the semigroup generated by $\frac{q_0}{s}$ and $q_0+1$ is $\frac{1}{2}(\frac{q_{0}}{s}-1)(q_{0}+1-1)$. Then the assertion follows from (\ref{n2}), together with Proposition \ref{genus}.
\end{proof}

Now some special cases for curves (C) are considered in greater detail.
\subsection{$q_{0}=2^{h}$, $s=2$, $H_s=\{0,1\}$}

Let $q_{0}$ be a power of $2$, and $s=2$. Since that $q_{0}$ is even, then $H=\FF_{q_{0}}$. Let $H_{s}=\{0,1\}$. Therefore, the rational map $\varphi : \cH\rightarrow\cX$ defined as in (\ref{quo}) is 
$$\varphi(x,y)=(x^{2}+x, y).$$
\begin{proposition}
$\cX$ has affine equation
\begin{equation}\label{eq1}
Y^{q_{0}+1}=X^{\frac{q_{0}}{2}}+X^{\frac{q_{0}}{4}}+\ldots+X^{2}+X.
\end{equation}
\end{proposition}
\begin{proof}
Let $(X,Y)\in\cH$. We need to prove that $\varphi(X,Y)$ satisfies (\ref{eq1}) for every point $(X,Y)\in \cH$. To do this it is enough to observe that
$$(X^{2}+X)^{\frac{q_{0}}{2}}+(X^{2}+X)^{\frac{q_{0}}{4}}+\ldots+(X^{2}+X)^{2}+(X^{2}+X)=X^{q_0}+X,$$
and take into account that $(X,Y)$ satisfies (\ref{Hcurve}).
\end{proof}
The curve with equation $\ref{eq1}$ was investigated in \cite{GV}.
\begin{proposition}[\cite{GV}]\label{sW2}
There exists a point $P\in\cX$ such that the Weierstrass semigroup at $P$ is
$$H(P)=\left\langle q_{0}-1, q_{0}, q_{0}+1\right\rangle.$$
\end{proposition}

Therefore, taking into account Proposition \ref{sW}, the curve $\cX$ has at least two different Weierstrass semigroups.



\subsubsection{$q_0=8$.} In this case the curve $\cX$ has equation
$$Y^{9}=X^{4}+X^{2}+X,$$
its genus is $g=12$, and the number of its $\FF_{64}$-rational points is $257$. In Appendix, we will denote by $(C1a)$ the Improved AG codes costructed from the Weierstrass semigroup of Proposition \ref{sW}, that is
$H(\bar{P}_{\infty})=\left\langle 4,9\right\rangle$, and by $(C1b)$ those costructed from the  Weierstrass semigroup of Proposition \ref{sW2}, $H(P)=\left\langle 7, 8, 9\right\rangle$.


\subsection{$q_0=2^{h}$, $s=\frac{q_0}{2}$, $H_{s}=\left\{a\in H |\,Tr(a)=0\right\}$.}

Let $q_{0}$ be a power of $2$, and $s=\frac{q_{0}}{2}$. Since that $q_{0}$ is even, then $H=\FF_{q_{0}}$. Let $H_{s}=\left\{a\in H |\,Tr(a)=0\right\}$,
where $Tr(a)=a+a^{2}+\ldots+a^{\frac{q_{0}}{4}}+a^{\frac{q_{0}}{2}}$.
Therefore, the rational map $\varphi : \cH\rightarrow\cX$ defined as in (\ref{quo}) is 
$$\varphi(x,y)=(\prod_{a\in\FF_{q_{0}}:\,Tr(a)=0}(x+a), y)=(Tr(x), y).$$
\begin{proposition}
The curve $\cX$ has affine equation
\begin{equation}\label{eq2}
Y^{q_{0}+1}=X^{2}+X.
\end{equation}
\end{proposition}
\begin{proof}
Let $(X,Y)\in\cH$. The assertion follows from the equation
$$(Tr(X))^{2}+(Tr(X))=X^{q_0}+X.$$
\end{proof}
\begin{proposition}[\cite{GV}]\label{sW3}
There exists a point $P\in\cX$ such that the Weierstrass semigroup at $P$ is
$$H(P)=\left\langle q_{0}+1-\frac{q_{0}}{2}, q_{0}+1-(\frac{q_{0}}{2}-1),\ldots, q_{0}, q_{0}+1\right\rangle.$$
\end{proposition}




\subsubsection{$q_{0}=16$.}   
The curve $\cX$ has equation
$$Y^{17}=X^{2}+X,$$
its genus is $g=8$, and the number of its $\FF_{256}$-rational points is $513$. In Appendix, the Improved AG codes costructed from the Weierstrass semigroup of Proposition \ref{sW}, that is
$H(\bar{P}_{\infty})=\left\langle 2, 17\right\rangle$, will be referred to as codes (C2).

\subsection{$q_{0}=16$, $s=4$, $H_s=\FF_{4}$.} 

Let $q_{0}=16$, and $s=4$. Since that $q_{0}$ is even, then $H=\FF_{16}$. So, we can consider the case $H_{s}=\FF_{4}$.
The rational map $\varphi : \cH\rightarrow\cX$ defined as in (\ref{quo}) is 
$$\varphi(x,y)=(x^{4}+x, y).$$
\begin{proposition}
The curve $\cX$ has affine equation
\begin{equation}\label{eq3}
Y^{17}=X^{4}+X.
\end{equation}
Moreover, $\cX$ has genus $g=24$, and the number of its $\FF_{256}$-rational points is $1025$.
\end{proposition}
\begin{proof}
Let $(X,Y)\in\cH$. We need to prove that $\varphi(X,Y)$ satisfies (\ref{eq3}). To do this it is enough to observe that
$$(X^{4}+X)^{4}+(X^{4}+X)=X^{16}+X,$$
and take into account that $(X,Y)$ satisfies (\ref{Hcurve}) for $q_0=16$.
The second part of the assertion follows from Proposition \ref{genus}.
\end{proof}

In Appendix, the Improved AG codes costructed from the Weierstrass semigroup of Proposition \ref{sW}, that is
$H(\bar{P}_{\infty})=\left\langle 4, 17\right\rangle$, will be referred to as codes (C3).

\subsection{$q_{0}=9$, $s=3$, $H_{s}=\left\{x\in H |\,x^{3}=\alpha x\right\}$, $\alpha$ primitive element of $\FF_{9}$.}

Let $q_{0}=9$, $s=3$. Also, let $\alpha$ be a primitive element of $\FF_{9}$. Note that $\alpha4=-1$.  Let $H_{s}=\left\{x\in H |\,x^{3}=\alpha x\right\}$. It is a straightforward computation to check that  $H_{s}\subset H$.  In fact, for $x\in H_s$,
$$x^{9}+x=(\alpha x)^{3}+x=\alpha^{3}x^{3}+x=\alpha^{4}x+x=-x+x=0.$$
The rational map $\varphi : \cH\rightarrow\cX$ defined as in (\ref{quo}) is 
$$\varphi(x,y)=(x^{3}-\alpha x, y).$$
\begin{proposition}
The curve $\cX$ has affine equation
\begin{equation}\label{eq4}
Y^{10}=X^{3}+\alpha^{3}X.
\end{equation}
Moreover, $\cX$ has genus $g=9$, and the number of its $\FF_{81}$-rational points is $244$.
\end{proposition}
\begin{proof}
Let $(X,Y)\in\cH$. We need to prove that $\varphi(X,Y)$ satisfies (\ref{eq4}). To do this it is enough to observe that
$$(X^{3}-\alpha X)^{3}+\alpha^{3}(X^{3}-\alpha X)=X9-\alpha^4X=X9+X,$$
and take into account  that $(X,Y)$ satisfies (\ref{Hcurve}) for $q_0=9$.
The second part of the assertion follows from Proposition \ref{genus}.
\end{proof}

In Appendix, the improved AG codes costructed from the Weierstrass semigroup  $H(\bar{P}_{\infty})=\left\langle 3, 10\right\rangle$ will be referred to as codes $(C4)$.

\section{Weierstrass semigroups for  curves (D)}

Let $m$ be a divisor of $q_{0}+1$, and
let $\cX_{m}$ be the non-singular model of the curve over $\FF_{q}$ with affine equation
$$Y^{m}=X^{q_{0}}+X.$$

\begin{proposition}\label{n3} The curve
$\cX_{m}$ has the following properties.
\begin{enumerate}
\item[\rm{(i)}] $\cX_{m}$ is maximal.
\item[\rm{(ii)}] The genus of $\cX_{m}$ is $g=\frac{1}{2}(q_{0}-1)(m-1)$.
\item[\rm{(iii)}] There exists precisely one $\FF_{q}$-rational point of $\cX_m$ centred at the only point at infinity of the plane curve  
$Y^{m}=X^{q_{0}}+X$, and
$$ord_{\bar{P}_{\infty}}(x)=-m,\,\,\,\,\,\,\,\,\,ord_{\bar{P}_{\infty}}(y)=-q_{0}$$
hold.
\end{enumerate}
\end{proposition}
\begin{proof}
Assertion (i) and (ii) follows from \cite[(IV) of Proposition 2.1]{CKT1}.
It is straightforward to check that
$\cX_{m}$ is the  image of the Hermitian curve by the rational map
$$\varphi : \cH\rightarrow\cX_{m},\,\,\,\,\,\,\,\,(x,y)\mapsto(x,y^{h}),$$
where $h=\frac{q_{0}+1}{m}$.
The only point at infinity $\bar{P}_{\infty}$ of $\cX_{m}$ is the image of $P_{\infty}\in\cH$ by $\varphi$, and it is easily seen that $e_{P_{\infty}}$ is equal to $deg(\varphi)=h$.
Let $f\in\bar{\FF}_{q}(x,y^{h})$, and let $\varphi^{*}$ be the pull-back of $\varphi$. Then
$ord_{P_{\infty}}(\varphi^{*}(f))=e_{P_{\infty}}ord_{\bar{P}_{\infty}}(f).$
Therefore,
$$\frac{q_{0}+1}{m}ord_{\bar{P}_{\infty}}(x)=ord_{P_{\infty}}(x)=-(q_{0}+1),$$
and
$$\frac{q_{0}+1}{m}ord_{\bar{P}_{\infty}}(y)=ord_{P_{\infty}}(y^{\frac{q_{0}+1}{m}})=\frac{q_{0}+1}{m}ord_{P_{\infty}}(y)=-\frac{q_{0}+1}{m}q_{0}.$$
Hence,
$$ord_{\bar{P}_{\infty}}(x)=-m,\textrm{ and }ord_{\bar{P}_{\infty}}(y)=-q_{0}.$$
\end{proof}


\begin{proposition}\label{sW5}
$H(\bar{P}_{\infty})=\left\langle m, q_{0}\right\rangle$.
\end{proposition}
\begin{proof}
 Note that ${q_0}$ and $m$ are coprime. Then by \cite[Proposition 5.33]{HLP} the genus of the semigroup generated by ${q_0}$ and $m$ is $\frac{1}{2}(q_0-1)(m-1)$. Then the assertion follows from  Proposition \ref{n3}.

\end{proof}

\begin{proposition}[\cite{GV}]\label{sW6}
There exists a point $P\in\cX_{m}$ such that
$$H(P)=\left\langle q_{0}+1-\frac{q_{0}+1}{m}, q_{0}+1-(\frac{q_{0}+1}{m}-1), \ldots, q_{0}+1 \right\rangle.$$
\end{proposition}

Now we consider the curve $\cX_{m}$ for particular values of $q_0$ and $m$.

\subsection{$q_{0}=7$, $m=4$.}  

The curve $\cX_{m}$ has equation
$$Y^{4}=X^{7}+X,$$
its genus is $g=9$, and the number of its $\FF_{49}$-rational points is $176$. In Appendix,  $(D1a)$ will denote the Improved AG codes costructed from the Weierstrass semigroup of Proposition \ref{sW5}, that is
$H(\bar{P}_{\infty})=\left\langle 4,7\right\rangle$, and  $(D1b)$ those costructed from the  Weierstrass semigroup of Proposition \ref{sW6}, $H(P)=\left\langle 6, 7, 8\right\rangle$.

\subsection{$q_{0}=7$, $m=2$.}  

The curve $\cX_{m}$ has equation
$$Y^{2}=X^{7}+X,$$
its genus is $g=3$, and the number of its $\FF_{49}$-rational points is $92$. In Appendix, $(D2a)$ will denote the Improved AG codes costructed from the Weierstrass semigroup of Proposition \ref{sW5}, that is
$H(\bar{P}_{\infty})=\left\langle 2,7\right\rangle$, and $(D2b)$ those costructed from the  Weierstrass semigroup of Proposition \ref{sW6}, $H(P)=\left\langle 4, 5, 6, 7\right\rangle$.

\subsection{$q_{0}=8$, $m=3$.}

The curve $\cX_{m}$ has equation
$$Y^{3}=X^{8}+X,$$
its genus is $g=7$, and the number of its $\FF_{64}$-rational points is $177$. In Appendix, $(D3a)$ will denote the improved AG codes costructed from the Weierstrass semigroup of Proposition \ref{sW5}, that is
$H(\bar{P}_{\infty})=\left\langle 3, 8\right\rangle$, and $(D3b)$ those costructed from the  Weierstrass semigroup of Proposition \ref{sW6}, $H(P)=\left\langle 6, 7, 8, 9\right\rangle$.

\subsection{$q_{0}=9$, $m=5$.}

The curve $\cX_{m}$ has equation
$$Y^{5}=X^{9}+X,$$
its genus is $g=16$, and the number of its $\FF_{81}$-rational points is $370$. In Appendix, $(D4a)$ will denote the Improved AG codes costructed from the Weierstrass semigroup of Proposition \ref{sW5}, that is
$H(\bar{P}_{\infty})=\left\langle 5, 9\right\rangle$, and $(D4b)$  those costructed from the  Weierstrass semigroup of Proposition \ref{sW6}, $H(P)=\left\langle 8, 9, 10\right\rangle$.

\section{Weierstrass semigroups for  curves (E)}

Let $m$ be a divisor of $q-1$ and
let $\cX_{m}$ be the non-singular model of the plane curve
\begin{equation}\label{eq5}
Y^{\frac{q-1}{m}}=X(X+1)^{q_{0}-1}.
\end{equation}

\begin{proposition}
The curve $\cX_m$ is the
 image of the Hermitian curve by the rational map
\begin{equation}\label{quo2}
\varphi : \cH\rightarrow\cX_{m},\,\,\, (x,y)\mapsto (t,z)=(x^{q_{0}-1}, y^{m}).
\end{equation}
\end{proposition}
\begin{proof}
Let $(X,Y)$ be a point in $\cH$. We need to prove that $\varphi(X,Y)$ satisfies (\ref{eq5}). This follows from
$$X^{q_{0}-1}(X^{q_{0}-1}+1)^{q_{0}-1}=(X^{q_{0}}+X)^{q_{0}-1}=(Y^{q_{0}+1})^{q_{0}-1}=(Y^{m})^{\frac{q-1}{m}}.$$
\end{proof}

It is easily seen that the rational functions $t=x^{q_{0}-1}$, and $z=y^{m}$ have just a pole $\bar{P}_{\infty}$, which is the image  by $\varphi$ of the only infinite point $P_{\infty}$ of $\cH$. Therefore, the ramification index $e_{P_{\infty}}$ is equal to $deg(\varphi)$.

Some properties of the curve $\cX_m$ were investigated in \cite[Corollary 4.9 and Example 6.3]{GSX}.

\begin{proposition}[\cite{GSX}]\label{grado} The genus of $\cX_m$ is equal to 
$$
\frac{1}{2m}(q_0-1)(q_0+1-d),
$$
where $d=(m,q_0+1)$.
The degree $deg(\varphi)$ of the rational map $\varphi$ is equal to $m$.
\end{proposition}

\begin{proposition}[\cite{GSX}] Let $\alpha$ be a primitive $m$-th root of unity in $\FF_{q}$, and let
$$G_{m}=\left\langle \phi :\cH\rightarrow\cH\, |\,\, (X,Y)\mapsto(\alpha^{q_{0}+1}, \alpha Y)\right\rangle.$$
 Then $\cX_{m}$ can be seen as the quotient curve of $\cH$ by the group $G_{m}$.
\end{proposition}

Note that a point $(a,b)$ of the plane curve (\ref{eq5}) is non-singular provided that $b\neq 0$. Let $P_{(a,b)}$ be the only point of $\cX_m$
lying over at $(a,b)$.
\begin{proposition} Let $b\neq 0$. Then the size of 
$\varphi^{-1}(P_{(a,b)})$ is equal to $m$.
\end{proposition}
\begin{proof}
Clearly, $\varphi^{-1}(P_{(a,b)})=\left\{(X,Y)\in\cH |\,X^{q_{0}-1}=a, Y^{m}=b\right\}$.
The number of roots of $Y^{m}-b$ is $m$, for all $b\neq0$, since that char$(\bar{\FF}_{q})\nmid m$. Hence, $\#\varphi^{-1}(a,b)\geq m$, when $b\neq0$. Now taking into account Proposition \ref{grado}, we have also that $\#\varphi^{-1}(a,b)\leq m$. Then the assertion follows.
\end{proof}

Let $(a,0)$ be a point of the plane curve (\ref{eq5}).  If $a=0$, then $(a,0)$ is non-singular. Let $P_{(0,0)}$ be the only point of $\cX_m$ lying over $(0,0)$. Clearly,  $\varphi^{-1}(P_{(0,0)})=\left\{P_{0}=(0,0)\in \cH\right\}$. Hence, $e_{P_{0}}=deg(\varphi)=m$.
Assume now that $a\neq0$. Then $a=-1$. We consider the set $L$ of points of $\cH$ whose image by $\varphi$ is a point of $\cX_m$ lying over $(-1,0)$. Clearly, $L=\left\{(X,Y)\in\cH |\,X^{q_{0}-1}=-1, Y^{m}=0\right\}=\left\{(X,0)\in\cH |\,X^{q_{0}-1}+1=0\right\}$. Let $\left\{P_{1},P_{2},\ldots,P_{q_{0}-1}\right\}$ be the points in $\cH$ such that $Y=0$ and $X^{q_{0}-1}+1=0$.

\begin{lemma}\label{ram}
Let $\varphi$ be as in {\rm (\ref{quo2})}. The ramification points of $\varphi$ are $P_{0},P_{\infty},P_{1},P_{2},\ldots,P_{q_{0}-1}$.
In particular, $e_{P_{0}}=e_{P_{\infty}}=m$, and $e_{P_{i}}=d$, for $1\leq i\leq q_{0}-1$, where $d=(m, q_{0}+1)$.
\end{lemma}
\begin{proof} By the previous results, we only need to calculate $e_{P_{i}}$.
Being $\cX\cong\cH/G_{m}$, the integer $e_{P}$ represents the stabilizer of $G_{m}$ at $P$, for $P\in\cH$; i.e. $e_{P}=\# Stab_{P}(G_{m})$.
Note that $\#Stab_{(X,0)}(G_{m})=\#\left\{0\leq i <m |\, \alpha^{i(q_{0}+1)}=1\right\}=(m, q_{0}+1)$. Hence, $$e_{P_{i}}=d,$$ where $d=(m, q_{0}+1)$.
\end{proof}

A straightforward corollary to Lemma \ref{ram} is that  there are exactly $\frac{d(q_{0}-1)}{m}$ points of $\cX_m$, say $\bar{P}_{1},\bar{P}_{2},\ldots,\bar{P}_{\frac{d(q_{0}-1)}{m}}$, lying over $(-1, 0)$.

\begin{proposition}
Let $D=\bar{P}_{1}+\bar{P}_{2}+\ldots+\bar{P}_{\frac{d(q_{0}+1)}{m}}$. Then
$$(z)=\frac{m}{d}D+\bar{P}_{0}-q_{0}\bar{P}_{\infty},\,\,\,\,\,\,\,(t+1)=\frac{q_{0}+1}{d}D-\frac{q-1}{m}\bar{P}_{\infty}.$$
\end{proposition}
\begin{proof}
The assertion follows from the following straightforward computation:
\begin{enumerate}
\item[$\bullet$] $$ord_{\bar{P}_{0}}(z)=\frac{1}{m}ord_{P_{0}}(y^{m})=ord_{P_{0}}(y)=1;$$
$$ord_{\bar{P}_{0}}(t+1)=\frac{1}{m}ord_{P_{0}}(x^{q_{0}-1}+1)=0;$$
\item[$\bullet$] $$ord_{\bar{P}_{\infty}}(z)=\frac{1}{m}ord_{P_{\infty}}(y^{m})=-q_{0};$$
$$ord_{\bar{P}_{\infty}}(t+1)=\frac{1}{m}ord_{P_{\infty}}(x^{q_{0}-1}+1)=\frac{q_{0}-1}{m}ord_{P_{\infty}}(x)=-\frac{q_{0}-1}{m};$$
\item[$\bullet$] for $1\leq i\leq \frac{d(q_{0}-1)}{m}$,
$$ord_{\bar{P}_{i}}(z)=\frac{1}{d}ord_{P_{i}}(y^{m})=\frac{m}{d};$$
$$ord_{\bar{P}_{i}}(t+1)=\frac{1}{d}ord_{P_{i}}(x^{q_{0}-1}+1)=\frac{q_{0}+1}{d}.$$
\end{enumerate}
\end{proof}

\begin{proposition}\label{pp}
Let $i,j\in\N_{0}$ such that
$$i\geq j\frac{q_{0}+1}{m}.$$
Then $iq_{0}-j\frac{q-1}{m}\in H(\bar{P}_{\infty})$.
\end{proposition}
\begin{proof}
Let $\gamma=z^{i}(t+1)^{-j}$ and $D=\bar{P}_{1}+\bar{P}_{2}+\ldots+\bar{P}_{\frac{d(q_{0}+1)}{m}}$. Then
$$(\gamma)=\frac{im}{d}D+\bar{P}_{0}-iq_{0}\bar{P}_{\infty}-\frac{j(q_{0}+1)}{d}D+\frac{j(q-1)}{m}\bar{P}_{\infty}.$$
Therefore,
$$(\gamma)_{\infty}=(iq_{0}-j\frac{q-1}{m})\bar{P}_{\infty}.$$
\end{proof}


\subsection{$q_{0}=7$, $m=3$}   

The curve $\cX_{m}$ has equation
$$Y^{16}=X(X+1)^{6},$$
its genus is $g=7$, and the number of its $\FF_{49}$-rational points is $148$.
By Proposition \ref{pp} we have that $5, 7, 8$ are non-gaps at $\bar{P}_{\infty}$. Moreover, the genus of the numerical semigroup generated by these integers is equal to the genus of $\cX_{m}$. Hence, $H(\bar{P}_{\infty})=\left\langle 5, 7, 8\right\rangle$. In Appendix, $(E1)$ will denote the Improved AG codes costructed from $H(\bar{P}_{\infty})$.

\subsection{$q_{0}=9$, $m=4$}   

The curve $\cX_{m}$ has equation
$$Y^{20}=X(X+1)^{8},$$
its genus is $g=8$, and the number of its $\FF_{81}$-rational points is $226$.
By Proposition \ref{pp} we have that $5, 7, 9$ are non-gaps at $\bar{P}_{\infty}$. Moreover, the genus of the numerical semigroup generated by these integers is equal to the genus of $\cX_{m}$. Hence, $H(\bar{P}_{\infty})=\left\langle 5, 7, 9\right\rangle$. In Appendix, $(E2)$ will denote the Improved AG codes costructed from $H(\bar{P}_{\infty})$.

\subsection{ $q_{0}=16$, $m=5$}  

The curve $\cX_{m}$ has equation
$$Y^{51}=X(X+1)^{15},$$
its genus is $g=24$, and the number of its $\FF_{256}$-rational points is $1025$.
By Proposition \ref{pp} we have that $10, 13, 16, 17$ are non-gaps at $\bar{P}_{\infty}$. Moreover, the genus of the numerical semigroup generated by these integers is equal to the genus of $\cX_{m}$. Hence, $H(\bar{P}_{\infty})=\left\langle 10, 13, 16, 17\right\rangle$. 
In Appendix, $(E3)$ will denote the Improved AG codes costructed from $H(\bar{P}_{\infty})$.

\section*{Appendix: Improvements on MinT's tables}

In this Appendix, we consider the parameters of some of the codes ${\tilde C}_d(P)$ (see Definition \ref{n5}), where $P$ is a point of a curve $\cX$ belonging to one of the families (A)-(E).

We recall the following propagation rules.
\begin{proposition}[see Exercise 7 in \cite{TV}]\label{spoiling}${}$
\begin{enumerate}
\item[{\rm (i)}] If there exists a $q$-ary linear code of lenght $n$, dimension $k$ and minimum distance $d$, then for each non-negative integer $s< d$ there exists a $q$-ary linear code of length $n$, dimension $k$ and minimum distance $d-s$.

\item[{\rm (ii)}] If there exists a $q$-ary linear code of lenght $n$, dimension $k$ and minimum distance $d$, then for each non-negative integer $s< k$ there exists a $q$-ary linear code of length $n$, dimension $k-s$ and minimum distance $d$.
\item[{\rm (iii)}] If there exists a $q$-ary linear code of lenght $n$, dimension $k$ and minimum distance $d$, then for each non-negative integer $s<k$ there exists a $q$-ary linear code of length $n-s$, dimension $k-s$ and minimum distance $d$.
\end{enumerate}
\end{proposition}

The notation of Section \ref{sec2} is kept. By Theorem \ref{dist2}, together with both (i) and (ii) of Proposition \ref{spoiling}, a code ${\tilde C}_d(P)$ can be assumed to be an $[n,k,d]_q$ code with $n=\#\cX(\fq)-1$ and $k=n-r_d$.  Note that $r_d$ can be obtained from the Weierstrass semigroup $H(P)$ by straightforward computation.

The following tables provide a list of codes that, according to the online database MinT \cite{MINT}, have larger minimum distance with respect to the previously known codes with same dimension and same length. 
The value of $s$ in each entry means that the $[n-i,k-i,d]_q$ code obtained from ${\tilde C}_d(P)$ by applying 
the propagation rule (iii) of Proposition \ref{spoiling} has better parameters that the known codes for each $i\le s$. For the sake of completeness, the parameters $[n-s,k-s,d]$ appear in the tables.


\newpage


\begin{tiny}
\begin{table*}
\renewcommand{\arraystretch}{1.3}
\caption{Improvements on \cite{MINT} - $q=49$} \label{49}
\centering
\begin{tabular}{|c|c|c||c||c|c|c||c|||c|c|c||c||c|c|c||c|}
\hline
$n$ & $k$ & $d$ & \ref{spoiling}(iii) & $n-s$ & $k-s$ & $d$ & Code & $n$ & $k$ & $d$ & \ref{spoiling}(iii) & $n-s$ & $k-s$ & $d$ & Code\\
\hline
91 & 80 & 9 & s=10 & 81 & 70 & 9 & (D2a)(D2b) & 91 & 50 & 39 & s=10 & 81 & 40 & 39 & (D2a)(D2b)\\
91 & 79 & 10 & s=10 & 81 & 69 & 10 & (D2a)(D2b) & 91 & 49 & 40 & s=10 & 81 & 39 & 40 & (D2a)(D2b)\\
91 & 78 & 11 & s=10 & 81 & 68 & 11 & (D2a)(D2b) & 91 & 48 & 41 & s=10 & 81 & 38 & 41 & (D2a)(D2b)\\
91 & 77 & 12 & s=10 & 81 & 67 & 12 & (D2a)(D2b) & 91 & 47 & 42 & s=10 & 81 & 37 & 42 & (D2a)(D2b)\\
91 & 76 & 13 & s=10 & 81 & 66 & 13 & (D2a)(D2b) & 91 & 46 & 43 & s=10 & 81 & 36 & 43 & (D2a)(D2b)\\
91 & 75 & 14 & s=10 & 81 & 65 & 14 & (D2a)(D2b) & 91 & 45 & 44 & s=10 & 81 & 35 & 44 & (D2a)(D2b)\\
91 & 74 & 15 & s=10 & 81 & 64 & 15 & (D2a)(D2b) & 91 & 44 & 45 & s=10 & 81 & 34 & 45 & (D2a)(D2b)\\
91 & 73 & 16 & s=10 & 81 & 63 & 16 & (D2a)(D2b) & 91 & 43 & 46 & s=10 & 81 & 33 & 46 & (D2a)(D2b)\\
91 & 72 & 17 & s=10 & 81 & 62 & 17 & (D2a)(D2b) & 91 & 42 & 47 & s=10 & 81 & 32 & 47 & (D2a)(D2b)\\
91 & 71 & 18 & s=10 & 81 & 61 & 18 & (D2a)(D2b) & 91 & 41 & 48 & s=10 & 81 & 31 & 48 & (D2a)(D2b)\\
91 & 70 & 19 & s=10 & 81 & 60 & 19 & (D2a)(D2b) & 91 & 40 & 49 & s=10 & 81 & 30 & 49 & (D2a)(D2b)\\
91 & 69 & 20 & s=10 & 81 & 59 & 20 & (D2a)(D2b) & 91 & 39 & 50 & s=10 & 81 & 29 & 50 & (D2a)(D2b)\\
91 & 68 & 21 & s=10 & 81 & 58 & 21 & (D2a)(D2b) & 91 & 38 & 51 & s=10 & 81 & 28 & 51 & (D2a)(D2b)\\
91 & 67 & 22 & s=10 & 81 & 57 & 22 & (D2a)(D2b) & 91 & 37 & 52 & s=10 & 81 & 27 & 52 & (D2a)(D2b)\\
91 & 66 & 23 & s=10 & 81 & 56 & 23 & (D2a)(D2b) & 91 & 36 & 53 & s=10 & 81 & 26 & 53 & (D2a)(D2b)\\
91 & 65 & 24 & s=10 & 81 & 55 & 24 & (D2a)(D2b) & 147 & 129 & 12 & s=18 & 129 & 111 & 12 & (E1)\\
91 & 64 & 25 & s=10 & 81 & 54 & 25 & (D2a)(D2b) & 147 & 128 & 13 & s=32 & 115 & 96 & 13 & (E1)\\
91 & 63 & 26 & s=10 & 81 & 53 & 26 & (D2a)(D2b) & 147 & 127 & 14 & s=32 & 115 & 95 & 14 & (E1)\\
91 & 62 & 27 & s=10 & 81 & 52 & 27 & (D2a)(D2b) & 147 & 126 & 15 & s=44 & 103 & 82 & 15 & (E1)\\
91 & 61 & 28 & s=10 & 81 & 51 & 28 & (D2a)(D2b) & 147 & 125 & 16 & s=46 & 101 & 79 & 16 & (E1)\\
91 & 60 & 29 & s=10 & 81 & 50 & 29 & (D2a)(D2b) & 147 & 124 & 17 & s=58 & 89 & 66 & 17 & (E1)\\
91 & 59 & 30 & s=10 & 81 & 49 & 30 & (D2a)(D2b) & 147 & 123 & 18 & s=58 & 89 & 65 & 18 & (E1)\\
91 & 58 & 31 & s=10 & 81 & 48 & 31 & (D2a)(D2b) & 147 & 122 & 19 & s=58 & 89 & 64 & 19 & (E1)\\
91 & 57 & 32 & s=10 & 81 & 47 & 32 & (D2a)(D2b) & 147 & 121 & 20 & s=58 & 89 & 63 & 20 & (E1)\\
91 & 56 & 33 & s=10 & 81 & 46 & 33 & (D2a)(D2b) & 147 & 120 & 21 & s=58 & 89 & 62 & 21 & (E1)\\
91 & 55 & 34 & s=10 & 81 & 45 & 34 & (D2a)(D2b) & 147 & 119 & 22 & s=58 & 89 & 61 & 22 & (E1)\\
91 & 54 & 35 & s=10 & 81 & 44 & 35 & (D2a)(D2b) & 147 & 118 & 23 & s=58 & 89 & 60 & 23 & (E1)\\
91 & 53 & 36 & s=10 & 81 & 43 & 36 & (D2a)(D2b) & 147 & 117 & 24 & s=58 & 89 & 59 & 24 & (E1)\\
91 & 52 & 37 & s=10 & 81 & 42 & 37 & (D2a)(D2b) & 147 & 116 & 25 & s=58 & 89 & 58 & 25 & (E1)\\
91 & 51 & 38 & s=10 & 81 & 41 & 38 & (D2a)(D2b) & 147 & 115 & 26 & s=58 & 89 & 57 & 26 & (E1)\\
\hline
\end{tabular}
\end{table*}
\end{tiny}

\newpage

\begin{tiny}
\begin{table*}
\renewcommand{\arraystretch}{1.3}
\caption{Improvements on \cite{MINT} - $q=49$} \label{49bis}
\centering
\begin{tabular}{|c|c|c||c||c|c|c||c|||c|c|c||c||c|c|c||c|}
\hline
$n$ & $k$ & $d$ & \ref{spoiling}(iii) & $n-s$ & $k-s$ & $d$ & Code & $n$ & $k$ & $d$ & \ref{spoiling}(iii) & $n-s$ & $k-s$ & $d$ & Code\\
\hline
147 & 114 & 27 & s=58 & 89 & 56 & 27 & (E1) & 175 & 148 & 19 & s=72 & 103 & 76 & 19 & (D1a)(D1b)\\
147 & 113 & 28 & s=58 & 89 & 55 & 28 & (E1) & 175 & 147 & 20 & s=74 & 101 & 73 & 20 & (D1a)(D1b)\\
147 & 112 & 29 & s=58 & 89 & 54 & 29 & (E1) & 175 & 146 & 21 & s=82 & 93 & 64 & 21 & (D1a)(D1b)\\
147 & 111 & 30 & s=58 & 89 & 53 & 30 & (E1) & 175 & 145 & 22 & s=82 & 93 & 63 & 22 & (D1a)(D1b)\\
147 & 110 & 31 & s=58 & 89 & 52 & 31 & (E1) & 175 & 144 & 23 & s=82 & 93 & 62 & 23 & (D1a)(D1b)\\
147 & 109 & 32 & s=58 & 89 & 51 & 32 & (E1) & 175 & 143 & 24 & s=82 & 93 & 61 & 24 & (D1a)(D1b)\\
147 & 108 & 33 & s=58 & 89 & 50 & 33 & (E1) & 175 & 142 & 25 & s=82 & 93 & 60 & 25 & (D1a)(D1b)\\
147 & 107 & 34 & s=58 & 89 & 49 & 34 & (E1) & 175 & 141 & 26 & s=82 & 93 & 59 & 26 & (D1a)(D1b)\\
147 & 106 & 35 & s=58 & 89 & 48 & 35 & (E1) & 175 & 140 & 27 & s=82 & 93 & 58 & 27 & (D1a)(D1b)\\
147 & 105 & 36 & s=58 & 89 & 47 & 36 & (E1) & 175 & 139 & 28 & s=82 & 93 & 57 & 28 & (D1a)(D1b)\\
147 & 104 & 37 & s=58 & 89 & 46 & 37 & (E1) & 175 & 138 & 29 & s=82 & 93 & 56 & 29 & (D1a)(D1b)\\
147 & 103 & 38 & s=58 & 89 & 45 & 38 & (E1) & 175 & 137 & 30 & s=82 & 93 & 55 & 30 & (D1a)(D1b)\\
147 & 102 & 39 & s=58 & 89 & 44 & 39 & (E1) & 175 & 136 & 31 & s=82 & 93 & 54 & 31 & (D1a)(D1b)\\
147 & 101 & 40 & s=58 & 89 & 43 & 40 & (E1) & 175 & 135 & 32 & s=82 & 93 & 53 & 32 & (D1a)(D1b)\\
147 & 100 & 41 & s=58 & 89 & 42 & 41 & (E1) & 175 & 134 & 33 & s=82 & 93 & 52 & 33 & (D1a)(D1b)\\
147 & 99 & 42 & s=58 & 89 & 41 & 42 & (E1) & 175 & 133 & 34 & s=82 & 93 & 51 & 34 & (D1a)(D1b)\\
147 & 98 & 43 & s=58 & 89 & 40 & 43 & (E1) & 175 & 132 & 35 & s=82 & 93 & 50 & 35 & (D1a)(D1b)\\
147 & 97 & 44 & s=58 & 89 & 39 & 44 & (E1) & 175 & 131 & 36 & s=82 & 93 & 49 & 36 & (D1a)(D1b)\\
147 & 96 & 45 & s=58 & 89 & 38 & 45 & (E1) & 175 & 130 & 37 & s=82 & 93 & 48 & 37 & (D1a)(D1b)\\
147 & 95 & 46 & s=58 & 89 & 37 & 46 & (E1) & 175 & 129 & 38 & s=82 & 93 & 47 & 38 & (D1a)(D1b)\\
147 & 94 & 47 & s=58 & 89 & 36 & 47 & (E1) & 175 & 128 & 39 & s=82 & 93 & 46 & 39 & (D1a)(D1b)\\
147 & 93 & 48 & s=58 & 89 & 35 & 48 & (E1) & 175 & 127 & 40 & s=82 & 93 & 45 & 40 & (D1a)(D1b)\\
147 & 92 & 49 & s=58 & 89 & 34 & 49 & (E1) & 175 & 126 & 41 & s=82 & 93 & 44 & 41 & (D1a)(D1b)\\
175 & 157 & 12 & s=46 & 129 & 111 & 12 & (D1a)(D1b) & 175 & 125 & 42 & s=82 & 93 & 43 & 42 & (D1a)(D1b)\\
175 & 155 & 13 & s=25 & 150 & 130 & 13 & (D1b) & 175 & 124 & 43 & s=82 & 93 & 42 & 43 & (D1a)(D1b)\\
175 & 154 & 14 & s=24 & 151 & 130 & 14 & (D1a) & 175 & 123 & 44 & s=82 & 93 & 41 & 44 & (D1a)(D1b)\\
175 & 153 & 15 & s=60 & 115 & 93 & 15 & (D1a) & 175 & 122 & 45 & s=82 & 93 & 40 & 45 & (D1a)(D1b)\\
175 & 152 & 15 & s=22 & 153 & 130 & 15 & (D1b) & 175 & 121 & 46 & s=82 & 93 & 39 & 46 & (D1a)(D1b)\\
175 & 151 & 16 & s=46 & 129 & 105 & 16 & (D1a)(D1b) & 175 & 120 & 47 & s=82 & 93 & 38 & 47 & (D1a)(D1b)\\
175 & 150 & 18 & s=74 & 101 & 76 & 18 & (D1a)(D1b) &   &   &   &   &   &   &   &  \\
\hline
\end{tabular}
\end{table*}
\end{tiny}

\newpage


\begin{table*}
\renewcommand{\arraystretch}{1.3}
\caption{Improvements on \cite{MINT} - $q=64$} \label{64}
\centering
\begin{tabular}{|c|c|c||c||c|c|c||c|}
\hline
$n$ & $k$ & $d$ & Prop.\ref{spoiling}(iii) & $n-s$ & $k-s$ & $d$ & Code\\
\hline
176 & 162 & 10 & s=29 & 147 & 133 & 10 & (D3b)\\
176 & 159 & 12 & s=14 & 162 & 145 & 12 & (D3a)\\
176 & 157 & 14 & s=14 & 162 & 143 & 14 & (D3a)\\
256 & 232 & 15 & s=30 & 226 & 202 & 15 & (C1a)\\
256 & 231 & 16 & s=30 & 226 & 201 & 16 & (C1a)\\
256 & 230 & 16 & s=19 & 237 & 211 & 16 & (C1b)\\
256 & 229 & 18 & s=30 & 226 & 199 & 18 & (C1b)\\
256 & 228 & 18 & s=28 & 228 & 200 & 18 & (C1a)\\
256 & 226 & 20 & s=28 & 228 & 198 & 20 & (C1a)\\
256 & 225 & 21 & s=28 & 228 & 197 & 21 & (C1a)\\
256 & 222 & 24 & s=28 & 228 & 194 & 24 & (C1a)\\
\hline
\end{tabular}
\end{table*}

\newpage


\begin{table*}
\renewcommand{\arraystretch}{1.3}
\caption{Improvements on \cite{MINT} - $q=81$} \label{81}
\centering
\begin{tabular}{|c|c|c||c||c|c|c||c|}
\hline
$n$ & $k$ & $d$ & Prop.\ref{spoiling}(iii) & $n-s$ & $k-s$ & $d$ & Code\\
\hline
225 & 207 & 12 & s=24 & 201 & 183 & 12 & (B1)(E2)\\
243 & 225 & 12 & s=42 & 201 & 183 & 12 & (C4)\\
243 & 223 & 13 & s=16 & 227 & 207 & 13 & (C4)\\
243 & 222 & 14 & s=16 & 227 & 206 & 14 & (C4)\\
243 & 221 & 15 & s=16 & 227 & 205 & 15 & (C4)\\
243 & 220 & 16 & s=16 & 227 & 204 & 16 & (C4)\\
243 & 218 & 18 & s=16 & 227 & 202 & 18 & (C4)\\
369 & 339 & 18 & s=25 & 344 & 314 & 18 & (D4a)(D4b)\\
369 & 337 & 19 & s=4 & 365 & 333 & 19 & (D4a)\\
369 & 336 & 20 & s=36 & 333 & 300 & 20 & (D4a)\\
369 & 334 & 21 & s=28 & 341 & 306 & 21 & (D4a)\\
369 & 333 & 23 & s=66 & 303 & 267 & 23 & (D4a)(D4b)\\
369 & 332 & 24 & s=66 & 303 & 266 & 24 & (D4a)(D4b)\\
369 & 330 & 25 & s=64 & 305 & 266 & 25 & (D4b)\\
369 & 328 & 27 & s=64 & 305 & 264 & 27 & (D4a)\\
369 & 327 & 28 & s=64 & 305 & 263 & 28 & (D4a)\\
369 & 323 & 32 & s=64 & 305 & 259 & 32 & (D4a)(D4b)\\
\hline
\end{tabular}
\end{table*}

\newpage


\begin{tiny}
\begin{table*}
\renewcommand{\arraystretch}{1.3}
\caption{Improvements on \cite{MINT} - $q=256$} \label{256}
\centering
\begin{tabular}{|c|c|c||c||c|c|c||c|}
\hline
$n$ & $k$ & $d$ & Prop.\ref{spoiling}(iii) & $n-s$ & $k-s$ & $d$ & Code\\
\hline
512 & 495 & 14 & s=186 & 326 & 309 & 14 & (C2)\\
512 & 494 & 16 & s=188 & 324 & 306 & 16 & (C2)\\
512 & 493 & 17 & s=188 & 324 & 305 & 17 & (C2)\\
512 & 492 & 18 & s=188 & 324 & 304 & 18 & (C2)\\
512 & 491 & 19 & s=188 & 324 & 303 & 19 & (C2)\\
512 & 490 & 20 & s=188 & 324 & 302 & 20 & (C2)\\
512 & 489 & 21 & s=188 & 324 & 301 & 21 & (C2)\\
512 & 488 & 22 & s=188 & 324 & 300 & 22 & (C2)\\
512 & 487 & 23 & s=188 & 324 & 299 & 23 & (C2)\\
512 & 486 & 24 & s=188 & 324 & 298 & 24 & (C2)\\
512 & 485 & 25 & s=188 & 324 & 297 & 25 & (C2)\\
512 & 484 & 26 & s=188 & 324 & 296 & 26 & (C2)\\
512 & 483 & 27 & s=188 & 324 & 295 & 27 & (C2)\\
512 & 482 & 28 & s=188 & 324 & 294 & 28 & (C2)\\
512 & 481 & 29 & s=188 & 324 & 293 & 29 & (C2)\\
512 & 480 & 30 & s=188 & 324 & 292 & 30 & (C2)\\
512 & 479 & 31 & s=188 & 324 & 291 & 31 & (C2)\\
512 & 478 & 32 & s=188 & 324 & 290 & 32 & (C2)\\
512 & 477 & 33 & s=188 & 324 & 289 & 33 & (C2)\\
512 & 476 & 34 & s=188 & 324 & 288 & 34 & (C2)\\
512 & 475 & 35 & s=188 & 324 & 287 & 35 & (C2)\\
512 & 474 & 36 & s=188 & 324 & 286 & 36 & (C2)\\
512 & 473 & 37 & s=188 & 324 & 285 & 37 & (C2)\\
512 & 472 & 38 & s=188 & 324 & 284 & 38 & (C2)\\
512 & 471 & 39 & s=188 & 324 & 283 & 39 & (C2)\\
512 & 470 & 40 & s=188 & 324 & 282 & 40 & (C2)\\
512 & 469 & 41 & s=188 & 324 & 281 & 41 & (C2)\\
512 & 468 & 42 & s=188 & 324 & 280 & 42 & (C2)\\
512 & 467 & 43 & s=188 & 324 & 279 & 43 & (C2)\\
512 & 466 & 44 & s=188 & 324 & 278 & 44 & (C2)\\
512 & 465 & 45 & s=188 & 324 & 277 & 45 & (C2)\\
512 & 464 & 46 & s=188 & 324 & 276 & 46 & (C2)\\
512 & 463 & 47 & s=188 & 324 & 275 & 47 & (C2)\\
512 & 462 & 48 & s=188 & 324 & 274 & 48 & (C2)\\
512 & 461 & 49 & s=188 & 324 & 273 & 49 & (C2)\\
512 & 460 & 50 & s=188 & 324 & 272 & 50 & (C2)\\
512 & 459 & 51 & s=188 & 324 & 271 & 51 & (C2)\\
512 & 458 & 52 & s=188 & 324 & 270 & 52 & (C2)\\
\hline
\end{tabular}
\end{table*}
\end{tiny}

\begin{tiny}
\begin{table*}
\renewcommand{\arraystretch}{1.3}
\caption{Improvements on \cite{MINT} - $q=256$} \label{256bis}
\centering
\begin{tabular}{|c|c|c||c||c|c|c||c|}
\hline
$n$ & $k$ & $d$ & Prop.\ref{spoiling}(iii) & $n-s$ & $k-s$ & $d$ & Code\\
\hline
512 & 457 & 53 & s=188 & 324 & 269 & 53 & (C2)\\
512 & 456 & 54 & s=188 & 324 & 268 & 54 & (C2)\\
512 & 455 & 55 & s=188 & 324 & 267 & 55 & (C2)\\
512 & 454 & 56 & s=188 & 324 & 266 & 56 & (C2)\\
512 & 453 & 57 & s=188 & 324 & 265 & 57 & (C2)\\
512 & 452 & 58 & s=188 & 324 & 264 & 58 & (C2)\\
512 & 451 & 59 & s=188 & 324 & 263 & 59 & (C2)\\
512 & 450 & 60 & s=188 & 324 & 262 & 60 & (C2)\\
512 & 449 & 61 & s=188 & 324 & 261 & 61 & (C2)\\
512 & 448 & 62 & s=188 & 324 & 260 & 62 & (C2)\\
512 & 447 & 63 & s=188 & 324 & 259 & 63 & (C2)\\
512 & 446 & 64 & s=188 & 324 & 258 & 64 & (C2)\\
512 & 445 & 65 & s=188 & 324 & 257 & 65 & (C2)\\
512 & 444 & 66 & s=188 & 324 & 256 & 66 & (C2)\\
1024 & 980 & 27 & s=247 & 777 & 733 & 27 & (C3)\\
1024 & 979 & 28 & s=248 & 776 & 731 & 28 & (C3)\\
1024 & 978 & 30 & s=413 & 611 & 565 & 30 & (C3)\\
1024 & 976 & 32 & s=445 & 579 & 531 & 32 & (C3)\\
1024 & 975 & 33 & s=477 & 547 & 498 & 33 & (C3)\\
1024 & 974 & 35 & s=494 & 530 & 480 & 35 & (C3)\\
1024 & 973 & 36 & s=494 & 530 & 479 & 36 & (C3)\\
1024 & 972 & 40 & s=500 & 524 & 472 & 40 & (C3)\\
1024 & 970 & 41 & s=498 & 526 & 472 & 41 & (C3)\\
1024 & 969 & 42 & s=498 & 526 & 471 & 42 & (C3)\\
1024 & 968 & 43 & s=498 & 526 & 470 & 43 & (C3)\\
1024 & 967 & 44 & s=498 & 526 & 469 & 44 & (C3)\\
1024 & 966 & 45 & s=498 & 526 & 468 & 45 & (C3)\\
1024 & 965 & 46 & s=498 & 526 & 467 & 46 & (C3)\\
1024 & 964 & 47 & s=498 & 526 & 466 & 47 & (C3)\\
1024 & 963 & 48 & s=498 & 526 & 465 & 48 & (C3)\\
1024 & 962 & 49 & s=498 & 526 & 464 & 49 & (C3)\\
1024 & 961 & 50 & s=498 & 526 & 463 & 50 & (C3)\\
1024 & 960 & 51 & s=498 & 526 & 462 & 51 & (C3)\\
1024 & 959 & 52 & s=498 & 526 & 461 & 52 & (C3)\\
1024 & 958 & 53 & s=498 & 526 & 460 & 53 & (C3)\\
1024 & 957 & 54 & s=498 & 526 & 459 & 54 & (C3)\\
1024 & 956 & 55 & s=498 & 526 & 458 & 55 & (C3)\\
\hline
\end{tabular}
\end{table*}
\end{tiny}


\begin{thebibliography}{99}

\bibitem{CKT1} \textsc{Cossidente, A., Korchm\'aros, G. and Torres, F.}, Curves of large genus covered by the Hermitian curve,
{\em Comm. Algebra}, vol. 28, (2000), 4707--4728. 

\bibitem{FR} \textsc{Feng, G.L. and  Rao, T.R.N.}, A simple approach for construction of algebraic-geometric codes from affine plane curves, \emph{IEEE Trans. Inform. Theory}, vol. 40, (1994), 1003--1012.

\bibitem{FR2} \textsc{Feng, G.L. and  Rao, T.R.N.}, Improved geometric Goppa codes, Part I: Basic theory, \emph{IEEE Trans. Inform. Theory}, vol. 41, (1995), 1678--1693.


\bibitem{FGT} \textsc{Fuhrmann, R., Garcia, A. and Torres, F.},  On maximal curves, \emph{J. Number Theory}, vol. 67,~(1997),~29-51.

\bibitem{GSX} \textsc{Garcia, A., Stichtenoth, H. and Xing, C.P.}, On subfields of the Hermitian function field, \emph{Compositio Math.} {\bf 120} (2000), 137--170.

\bibitem{GV} \textsc{Garcia, A. and Viana, P.}, Weierstrass points on certain non-classic curves, \emph{Arch. Math.}, {\bf 46} (1986), 315--322.

\bibitem{GHKT} \textsc{Giulietti, M., Hirschfeld, J.W.P., Korchm\'aros, G. and Torres, F.}, A family of curves covered by the Hermitian curve, \emph{S\'emin. Congr, Soc. Math. France} {\bf 21} (2009), 63--78.

\bibitem{GHKT2} \textsc{Giulietti, M., Hirschfeld, J.W.P., Korchm\'aros, G. and Torres, F.}, Curves covered by the Hermitian curve. Available online at www.sciencedirect.com.

\bibitem{GKL} \textsc{Garcia, A., Kim, S.J. and Lax, R.F.}, Consecutive Weierstrass gaps and minimum distance of Goppa codes, \emph{J. Pure Appl. Algebra}, vol. 84,~(1993),~199-207.

\bibitem{GOP} \textsc{Goppa, V.D.},  Codes associated with divisors, {\em Problemi Peredachi Informatsii}, vol. 13, (1977), 33-39.

\bibitem{HKT}  \textsc{Hirschfeld, J.W.P., Korchm\'aros, G. and Torres, F.}, {\em Algebraic Curves over a Finite Field}, Princeton/Oxford: Princeton University Press, 2008.

\bibitem{HLP} \textsc{H{\o}holdt, T., van Lint, J.H. and Pellikaan, R.}, Algebraic Geometry codes, in \emph{Handbook of Coding Theory}, vol. 1,  V.S. Pless, W.C.Huffman and R.A. Brualdi Eds. Amsterdam: Elsevier,~(1998),~871-961.

\bibitem{MINT} MinT. (2009, January). Tables of optimal parameters for linear codes,
University of Salzburg. Available: {http://mint.sbg.ac.at/}.

\bibitem{SV} \textsc{St\"ohr, K.O. and Voloch, J.F.},  Weierstrass points and curves over finite fields, \emph{Proc. London Math. Soc. $(3)$}, vol. 52,~(1986),~1-19.

\bibitem{TV} \textsc{Tsfasman, M.A. and Vladut, S.G.}, {\em Algebraic-goemetric codes}, Amsterdam: Kluwer, 1991.

\end{thebibliography}
\end{document}